\documentclass[12pt,a4paper]{amsart}
\usepackage[a4paper,left=2.25cm,right=2.25cm,top=3cm,bottom=3cm,headsep=1cm]{geometry}
\usepackage{euler} 
\usepackage{amssymb}
\newtheorem{Theorem}{Theorem}
\newtheorem{Corollary}[Theorem]{Corollary}
\newtheorem{Proposition}[Theorem]{Proposition}
\newtheorem{Lemma}[Theorem]{Lemma}
\newtheorem{Example}[Theorem]{Example}
\DeclareMathOperator\csm{csm}
\DeclareMathOperator\ssm{ssm}
\DeclareMathOperator\init{init}

\usepackage{lipsum}
\usepackage{tikz}\usetikzlibrary{svg.path,shapes.geometric}
\usepackage{hyperref}
\renewcommand\ss{\scriptstyle}

\title{Generic pipe dreams, lower-upper varieties, and Schwartz--MacPherson classes}

\author{Allen Knutson}
\address{Allen Knutson, Cornell University, Ithaca, New York}
\email{allenk@math.cornell.edu}
\author{Paul Zinn-Justin}
\address{Paul Zinn-Justin, School of Mathematics and Statistics, The University of Melbourne, 
Victoria 3010, Australia}
\email{pzinn@unimelb.edu.au}

\begin{document}
\begin{abstract}We recall the {\em lower-upper varieties} $E_w$ from [Knutson '05] 
  and give a formula for their equivariant cohomology classes, as a sum
  over {\em generic pipe dreams}. We recover as limits the classic
  and bumpless pipe dream formul\ae\ for double Schubert polynomials.
  As a byproduct, we obtain a formula for the degree of the $n$th
  commuting variety as a sum of powers of $2$.

  Generic pipe dreams also appear in the Segre--Schwarz--MacPherson
  analogue of the AJS/Billey formula, and when computing the Chern--Schwarz--MacPherson class
  of the orbit $B_- w B_+ \subseteq Mat_{k\times n}$
  or of a double Bruhat cell $B_-u B_+ \cap B_+ v B_-$.
\end{abstract}

\def\rem#1{}

\maketitle

\newcommand\lie[1]{{\mathfrak #1}}
\newcommand\into\hookrightarrow
\newcommand\onto\twoheadrightarrow
\newcommand\union\cup

\newcommand\defn[1]{{\bf #1}}
\newcommand\ZZ{{\mathbb Z}}
\newcommand\CC{{\mathbb C}}
\newcommand\RR{{\mathbb R}}
\newcommand\iso\cong
\newcommand\dom\backslash
\newcommand\barX{{\overline X}}
\newcommand\im{{\rm im}}
\newcommand\Tr{{\rm Tr}}
\newcommand\adj{{\rm adj}}
\newcommand\junk[1]{}

\newcommand\plaqctr[1]{\begin{tikzpicture}[baseline={([yshift=-\the\dimexpr\fontdimen22\textfont2\relax]current  bounding  box.center)}]\plaq{#1}\end{tikzpicture}}

%
%
%
\makeatletter
\newcommand{\gettikzxy}[3]{
  \tikz@scan@one@point\pgfutil@firstofone#1\relax
\pgfmathsetmacro{#2}{\the\pgf@x/\linkpatternunit}
\pgfmathsetmacro{#3}{\the\pgf@y/\linkpatternunit}
}
\tikzset{label anchor/.code={%
    \let\tikz@auto@anchor=\pgfutil@empty
    \def\tikz@anchor{#1}
  },
  label anchor/.default=center
}
\makeatother
%
\tikzset{arrow/.style={postaction={decorate,thick,decoration={markings,mark = at position #1 with {\arrow{>}}}}},arrow/.default=0.5}
\tikzset{invarrow/.style={postaction={decorate,thick,decoration={markings,mark = at position #1 with {\arrow{<}}}}},invarrow/.default=0.5}
\newdimen\linkpatternunit%
\newcount\linkpatternsize%
\newcount\lpsize
%
\newif\iflinkpatterninverted
\newif\iflinkpatterntikzstarted
\newif\iflinkpatternboxed
\newif\iflinkpatternaxis
\newif\iflinkpatternstraightlines
\newif\iflinkpatternnumbered
\newif\iflinkpatternalias
\newif\iflinkpatternnode
\newif\iflinkpatterncentered
\newcount\linkpatternfused
%
\pgfkeys{/linkpattern/.cd,centered/.is if=linkpatterncentered,inverted/.is if=linkpatterninverted,numbered/.is if=linkpatternnumbered,tikzstarted/.is if=linkpatterntikzstarted,straight lines/.is if=linkpatternstraightlines,boxed/.is if=linkpatternboxed,axis/.is if=linkpatternaxis,vertexcolor/.store in=\linkpatternvertexcolor,edgecolor/.store in=\linkpatternedgecolor,boxcolor/.code={\linkpatternboxedtrue\def\linkpatternboxcolor{#1}},tikzoptions/.style={every linkpattern/.append style={#1}},size/.code={\linkpatternsize=#1},numbering0/.code={\def\lpnumbering{#1}\def\linkpatternnumbering{#1}},numbering/.style={numbered,numbering0={#1}},unit/.code={\linkpatternunit=#1},height/.store in=\linkpatternheight,shape/.store in=\linkpatternshape,looseness/.store in=\linkpatternlooseness,squareness/.store in=\linkpatternsquareness,extra space/.store in=\linkpatternextraspace,width/.store in=\linkpatternwidth,alias/.is if=linkpatternalias,pos/.store in=\linkpatternpos,labeloptions/.style={labeloptionslist/.append style={#1}},labeloptionslist/.style={inner sep=2pt,font=\scriptsize,execute at begin node=$,execute at end node=$,label anchor={#1+180}},nodeon/.is if=linkpatternnode,node/.style={nodeon,nodeoptionslist/.append style={#1}},nodeoptionslist/.style={},
pipedream/.style={shape=pipedream,looseness=0,straight lines,numbering0=tangle},
tangle/.style={shape=tangle,numbering0=tangle},
every linkpattern/.style={x=\linkpatternunit,y=\linkpatternunit},
fused/.code={\linkpatternfused=#1},
%
vertex/.style={circle,thin,solid,draw=black,fill=\linkpatternvertexcolor,inner sep=1.5pt,draw opacity=1,transform shape},
edge/.style={very thick,solid,draw=\linkpatternedgecolor,draw opacity=1}}
%
\linkpatterncenteredfalse%
\linkpatterninvertedfalse%
\linkpatternnumberedfalse%
\linkpatterntikzstartedfalse%
\linkpatternboxedfalse%
\linkpatternaxistrue%
\linkpatternaliastrue%
\linkpatternunit=0.6cm%
\linkpatternsize=0%
\linkpatternfused=1%
\linkpatternstraightlinesfalse%
\def\linkpatternlooseness{0.2}
\def\linkpatternsquareness{0.35}
\def\linkpatternvertexcolor{red}%
\def\linkpatternedgecolor{blue}%
\def\linkpatternboxcolor{none}%
\def\linkpatternheight{0}
\def\linkpatternwidth{0}
\def\linkpatternshape{default}
\def\linkpatternnumbering{default}
\def\linkpatternpos{(0,0)}
\def\linkpatternextraspace{0}
\linkpatternnodefalse
%
%
\def\firstchar#1#2\empty{#1}%
\def\linkpatterndo#1#2{
\edef\param{\csname linkpattern#2\endcsname}
\edef\firstcharparam{\expandafter\firstchar\param\empty}
\expandafter\ifcat\firstcharparam a
\expandafter\ifx\csname linkpattern#1\param\endcsname\relax
\csname linkpattern#1unknown\endcsname
\else
\csname linkpattern#1\csname linkpattern#2\endcsname\endcsname
\fi
\else
\csname linkpattern#1unknown\endcsname
\fi
}%
\def\linkpatterncoorddefault{\xdef\lpcoordx{\x}\xdef\lpcoordy{0}\xdef\lpangle{90}}%
\def\linkpatterncoordtangle{\ifnum\x>\lphalfsize\pgfmathparse{\lpsize+1-\x}\xdef\lpcoordx{\pgfmathresult}\xdef\lpcoordy{\lpheight}\xdef\lpangle{270}\else\xdef\lpcoordx{\x}\xdef\lpcoordy{-\lpheight}\xdef\lpangle{90}\fi}
\def\linkpatterncoordcircle{%
 \pgfmathparse{180+(\x-1)*360/\lpsize}
 \xdef\lpangle{\pgfmathresult}
 \pgfmathparse{-cos(\lpangle)}
\xdef\lpcoordx{\pgfmathresult}
 \pgfmathparse{-sin(\lpangle)}
 \xdef\lpcoordy{\pgfmathresult}
}%
\def\linkpatterncoordpipedream{\ifnum\x>\lphalfsize\pgfmathparse{\lpsize+1-\x-0.5}\xdef\lpcoordx{\pgfmathresult}\xdef\lpcoordy{0}\xdef\lpangle{270}\else\pgfmathparse{0.5-\x}\xdef\lpcoordy{\pgfmathresult}\xdef\lpcoordx{0}\xdef\lpangle{0}\fi}
\def\linkpatterncoordrectangle{
\ifnum\x>\lptqsize
\pgfmathparse{\lpsize+1-\x-0.5}\xdef\lpcoordx{\pgfmathresult}\xdef\lpcoordy{0}\xdef\lpangle{270}
\else\ifnum\x>\lphalfsize
\pgfmathparse{\x-\lptqsize-0.5}\xdef\lpcoordy{\pgfmathresult}\xdef\lpcoordx{\linkpatternwidth}\xdef\lpangle{180}
\else\ifnum\x>\linkpatternheight
\pgfmathparse{\x-\linkpatternheight-0.5}\xdef\lpcoordx{\pgfmathresult}\xdef\lpcoordy{-\linkpatternheight}\xdef\lpangle{90}
\else
\pgfmathparse{0.5-\x}\xdef\lpcoordy{\pgfmathresult}\xdef\lpcoordx{0}\xdef\lpangle{0}
\fi\fi\fi
}%
\def\linkpatterncoordunknown{\message{link pattern: unknown shape}}%
%
\def\linkpatterndrawaxisdefault{\draw (1-\linkpatternextraspace,0) -- (\lpsize+\linkpatternextraspace,0);}%
\def\linkpatterndrawaxistangle{\draw (1-\linkpatternextraspace,-\lpheight) -- (\lphalfsize+\linkpatternextraspace,-\lpheight) (1-\linkpatternextraspace,\lpheight) -- (\lphalfsize+\linkpatternextraspace,\lpheight);}%
\def\linkpatterndrawaxiscircle{\draw (0,0) circle (1);}%
\def\linkpatterndrawaxispipedream{\draw (0,-\lphalfsize) -- (0,0) -- (\lphalfsize,0);}%
\def\linkpatterndrawaxisrectangle{\draw (0,-\linkpatternheight) rectangle (\linkpatternwidth,0);}%
\def\linkpatterndrawaxisunknown{\message{link pattern: unknown shape}}%
%
\def\linkpatterndrawboxdefault{\draw[fill=\linkpatternboxcolor] (0.5-\linkpatternextraspace,0) rectangle (\lpsize+0.5+\linkpatternextraspace,\lpheight);}%
\def\linkpatterndrawboxtangle{\draw[fill=\linkpatternboxcolor] (0.5-\linkpatternextraspace,-\lpheight) rectangle (\lphalfsize+0.5+\linkpatternextraspace,\lpheight);}
\def\linkpatterndrawboxcircle{\draw[fill=\linkpatternboxcolor] (0,0) circle (1);}%
\def\linkpatterndrawboxunknown{\message{link pattern: unknown shape}}%
\def\linkpatterndrawboxpipedream{\draw[fill=\linkpatternboxcolor] (0,-\lphalfsize) rectangle (\lphalfsize,0);}%
\def\linkpatterndrawboxrectangle{\draw[fill=\linkpatternboxcolor] (0,-\linkpatternheight) rectangle (\linkpatternwidth,0);}%
%
\def\linkpatternsetsizeunknown{
\global\lpsize=\linkpatternsize
\if\linkpatternheight0
\xdef\maxsep{0}
\foreach \x/\xx in \mylist%
{%
\edef\tempx{\withoutprime{\x}}
\edef\tempxx{\withoutprime{\xx}}
\pgfmathparse{max(\maxsep,abs(\tempx-\tempxx))}
\xdef\maxsep{\pgfmathresult}
}%
\pgfmathparse{0.25+0.8*\linkpatternsquareness*\maxsep}
\xdef\lpheight{\pgfmathresult}
\else
\xdef\lpheight{\linkpatternheight}
\fi
}
\def\linkpatternsetsizerectangle{
\pgfmathtruncatemacro{\tempsize}{2*\linkpatternwidth+2*\linkpatternheight}
\global\lpsize=\tempsize
\pgfmathtruncatemacro{\tempsize}{\linkpatternwidth+2*\linkpatternheight}
\xdef\lptqsize{\tempsize}
}%
%
\newcount\tempsize
\def\linkpatternrightmostunknown{
\global\lpsize=0
\global\tempsize=0
\foreach\x/\labx in \linkpatternnumbering
{
\edef\tempx{\withoutprime{\x}}
\ifnum\lpsize<\tempx\global\lpsize=\tempx\fi
\global\advance\tempsize by 1
}
\ifnum\tempsize>\lpsize\global\lpsize=\tempsize\fi
}%
\def\linkpatternrightmostdefault{
\global\lpsize=0
\global\tempsize=0
\foreach \x/\y in \mylist
{
\edef\tempx{\withoutprime{\x}}
\ifnum\lpsize<\tempx\global\lpsize=\tempx\fi
\ifx\x\y
\global\advance\tempsize by 1
\else
\edef\tempy{\withoutprime{\y}}
\ifnum\lpsize<\tempy\global\lpsize=\tempy\fi%
\global\advance\tempsize by 2
\fi
}
\ifnum\tempsize>\lpsize\global\lpsize=\tempsize\fi
}%
\def\linkpatternrightmosttangle{
\global\lpsize=0
\global\tempsize=0
\foreach \x/\y in \mylist
{
\edef\tempx{\withoutprime{\x}}
\ifnum\lpsize<\tempx\global\lpsize=\tempx\fi
\ifx\x\y
\global\advance\tempsize by 1
\else
\edef\tempy{\withoutprime{\y}}
\ifnum\lpsize<\tempy\global\lpsize=\tempy\fi%
\global\advance\tempsize by 2
\fi
}
\global\advance\lpsize by\lpsize
\ifnum\tempsize>\lpsize\global\lpsize=\tempsize\fi
}%
\def\linkpatternrightmosthalftangle{\linkpatternrightmosttangle}
%
\def\linkpatternnumberingdefault{\xdef\lpnumbering{1,...,\lpsize}}%
\def\linkpatternnumberingtangle{%
\xdef\lpnumbering{1,...,\lphalfsize}
\foreach\x in {1,...,\lphalfsize}
{\xdef\lpnumbering{\lpnumbering,\x'}}
}%
\def\linkpatternnumberinghalftangle{%
\xdef\lpnumbering{1,...,\lphalfsize}
\foreach\x in {1,...,\lphalfsize}
{\xdef\lpnumbering{\lpnumbering,\x'/}}
}%
\def\linkpatternnumberingunknown{%
}%
\newcommand\linkpattern[2][]{
{
\pgfkeys{/linkpattern/.cd,#1}
\edef\mylist{#2}
\def\primetest##1'{}%
\def\hasaprime##1{\expandafter\primetest##1''}
\def\internalwithoutprime##1'{##1}%
\def\withoutprime##1{\if\hasaprime##1 %
\expandafter\internalwithoutprime##1\else ##1\fi}%
\iflinkpatternnumbered%
\iflinkpatterninverted
\tikzset{/linkpattern/lbl/.style n args={3}{label={[/linkpattern/labeloptionslist=-##1,##3] ##1:##2}}}%
\else%
\tikzset{/linkpattern/lbl/.style n args={3}{label={[/linkpattern/labeloptionslist=##1,##3] ##1:##2}}}%
\fi%
\else%
\tikzset{/linkpattern/lbl/.style={}}%
\fi%
\tikzifinpicture{\linkpatterntikzstartedtrue%
\begin{scope}[shift=\linkpatternpos,/linkpattern/every linkpattern]
}{%
\linkpatterntikzstartedfalse%
\iflinkpatterncentered
\begin{tikzpicture}[baseline=(current  bounding  box.center),/linkpattern/every linkpattern]%
\else
\begin{tikzpicture}[baseline=0,/linkpattern/every linkpattern]%
\fi
}%
\begin{scope}[local bounding box=link pattern box]
\iflinkpatterninverted%
\begin{scope}[yscale=-1]%
\fi%
\linkpatterndo{setsize}{shape}
\ifnum\lpsize=0
\linkpatterndo{rightmost}{numbering}
\fi
\pgfmathtruncatemacro{\lphalfsize}{\lpsize/2}
\linkpatterndo{numbering}{numbering}
\iflinkpatternboxed
\linkpatterndo{drawbox}{shape}
\else
\iflinkpatternaxis
\linkpatterndo{drawaxis}{shape}
\fi
\fi
\foreach\xx/\xlab/\opt in \lpnumbering
{
\ifx\xlab\opt\def\opt{}\fi
\if\hasaprime\xx %
\pgfmathtruncatemacro{\xx}{\lpsize+1-\withoutprime{\xx}}
\fi
%
%
\ifnum\linkpatternfused>1
\pgfmathsetmacro{\x}{0.4*(0.5+\linkpatternfused*(0.5+floor((\xx-1)/\linkpatternfused)))+0.6*\xx}
\else
\def\x{\xx}
\fi
\linkpatterndo{coord}{shape}
\iflinkpatternalias\def\xlabb{\xlab}\else\def\xlabb{\xx}\fi
\path (\lpcoordx,\lpcoordy) coordinate[/linkpattern/vertex,/linkpattern/lbl={\lpangle+180}{\xlab}{\opt},alias=v\xlabb] (v\xx) ++(\lpangle:\linkpatternunit) coordinate[alias=vv\xlabb] (vv\xx); 
}
\foreach \a/\b/\c in \mylist
{
\if\hasaprime\a %
\pgfmathtruncatemacro{\a}{\lpsize+1-\withoutprime{\a}}
\fi
\ifx\b\c\def\c{}\fi
\draw[/linkpattern/edge]
\ifx\a\b
(v\a)
\c
--
++(0,\lpheight);
\else
\pgfextra{
\if\hasaprime\b %
\pgfmathtruncatemacro{\b}{\lpsize+1-\withoutprime{\b}}
\fi
\gettikzxy{(v\a)}{\ax}{\ay}
\gettikzxy{(v\b)}{\bx}{\by}
\gettikzxy{(vv\a)}{\axx}{\ayy}
\gettikzxy{(vv\b)}{\bxx}{\byy}
\pgfmathsetmacro{\dist}{sqrt((\ax-\bx)*(\ax-\bx)+(\ay-\by)*(\ay-\by))}
\pgfmathsetmacro{\abx}{(\axx-\ax)*\dist*\linkpatternsquareness+(\bx-\ax)*\linkpatternlooseness)}
\pgfmathsetmacro{\aby}{(\ayy-\ay)*\dist*\linkpatternsquareness+(\by-\ay)*\linkpatternlooseness)}
\pgfmathsetmacro{\bax}{(\bxx-\bx)*\dist*\linkpatternsquareness+(\ax-\bx)*\linkpatternlooseness)}
\pgfmathsetmacro{\bay}{(\byy-\by)*\dist*\linkpatternsquareness+(\ay-\by)*\linkpatternlooseness)}
}
(v\a)
\c
\iflinkpatternstraightlines
\pgfextra{
\pgfmathsetmacro{\t}{((\ax-\bx)*\bay-(\ay-\by)*\bax)/(\aby*\bax-\abx*\bay)}
\pgfmathsetmacro{\abx}{\t*\abx}
\pgfmathsetmacro{\aby}{\t*\aby}
}
[rounded corners=0.2\linkpatternunit] -- ++(\abx,\aby) -- (v\b);
\else
.. controls ++(\abx,\aby) and ++(\bax,\bay) .. 
\fi
(v\b);
\fi
}
\end{scope}
\iflinkpatternnode
\node[fit=(link pattern box),/linkpattern/nodeoptionslist] {};
\fi
\iflinkpatterninverted
\end{scope}
\fi
\iflinkpatterntikzstarted
\end{scope}
\else%
\end{tikzpicture}%
\fi%
}}%
%
%
\newcommand\tanglelinkpattern[3][]{%
{
\pgfkeys{/linkpattern/.cd,#1}
\iflinkpatterninverted
\begin{tikzpicture}[/linkpattern/every linkpattern,baseline=\linkpatternunit]%
\else
\begin{tikzpicture}[/linkpattern/every linkpattern,baseline=-\linkpatternunit]%
\fi
\linkpattern[#1,tikzstarted,numbered=false]{#3}
\pgfmathtruncatemacro{\lptempsize}{2*\linkpatternsize}
\iflinkpatterninverted
\begin{scope}[yshift=0.5*\linkpatternunit]
\else
\begin{scope}[yshift=-0.5*\linkpatternunit]
\fi
\linkpattern[tangle,#1,tikzstarted,size=\lptempsize,
numbering=halftangle,
height=0.5]{#2}
\end{scope}
\end{tikzpicture}%
}}
%
%
\newcommand\diag[4][]{%
\pgfkeys{/linkpattern/.cd,#1}
\iflinkpatterntikzstarted\else%
\begin{tikzpicture}[scale=0.5]
\fi%
\iflinkpatterninverted%
\begin{scope}[yscale=-1]%
\fi%
\draw (0,0) grid (#2,#3);
\edef\mylist{#4}
\foreach\y/\x/\z in \mylist
{
\ifx\x\z
\draw[decorate,decoration={zigzag,
amplitude=1pt,segment length=5pt}]
(\x-0.5,#3) -- (\x-0.5,\y-0.5) node[circle,fill=black,inner sep=2pt] {} -- (#2,\y-0.5);
\else
\node at (\x-0.5,\y-0.5) {$\z$};
\fi
}
\iflinkpatterninverted
\end{scope}
\fi
\iflinkpatterntikzstarted\else%
\end{tikzpicture}%
\fi%
}
%
\makeatletter
\tikzset{circle split part fill/.style  args={#1,#2}{%
 alias=tmp@name,
  postaction={%
    insert path={
     \pgfextra{%
     \pgfpointdiff{\pgfpointanchor{\pgf@node@name}{center}}%
                  {\pgfpointanchor{\pgf@node@name}{east}}%
     \pgfmathsetmacro\insiderad{\pgf@x}
      \fill[#1] (\pgf@node@name.base) ([xshift=-\pgflinewidth]\pgf@node@name.east) arc
                          (0:180:\insiderad-\pgflinewidth)--cycle;
      \fill[#2] (\pgf@node@name.base) ([xshift=\pgflinewidth]\pgf@node@name.west)  arc
                           (180:360:\insiderad-\pgflinewidth)--cycle;                    }}}}}  
 \makeatother
\tikzset{bdot/.style={circle,circle split,draw,circle split part fill={black,white},thin,inner sep=1pt}}%
\tikzset{wdot/.style={circle,circle split,draw,circle split part fill={white,black},thin,inner sep=1pt}}%
%
%
%
\newcommand\circlelinkpattern[2][]{
{
\pgfkeys{/linkpattern/.cd,#1}
\iflinkpatterntikzstarted\else%
\begin{tikzpicture}[/linkpattern/every linkpattern]%
\fi%
\iflinkpatterninverted%
\begin{scope}[yscale=-1]%
\fi%
\global\lpsize=\linkpatternsize
\edef\mylist{#2}
\foreach \x/\y in \mylist
{
\ifnum\x>\lpsize\global\lpsize=\x\fi
\ifnum\y>\lpsize\global\lpsize=\y\fi
}
%
\iflinkpatternaxis
\draw (0,0) circle (1);
\fi
\foreach\x in {1,...,\lpsize}
{
\pgfmathparse{(0.3*floor((\x-1)/\linkpatternfused)+0.7*((\x-0.5)/\linkpatternfused-0.5))*\linkpatternfused*360/\lpsize}
\coordinate[/linkpattern/vertex] (v\x) at (\pgfmathresult:1);
}
\foreach \x/\y/\z in \mylist
{
\ifx\y\z%
\draw[/linkpattern/edge] (v\x) .. controls ($0.5*(v\x)$) and  ($0.5*(v\y)$) .. (v\y);
\else
\draw[/linkpattern/edge] \z (v\x) .. controls ($0.5*(v\x)$) and  ($0.5*(v\y)$) .. (v\y);
\fi
}
\iflinkpatternnumbered%
\pgfmathparse{\lpsize/\linkpatternfused}
\global\lpsize=\pgfmathresult
\def\linkpatternnumbering{1,...,\lpsize}
\newdimen\angle
\foreach\x/\xx/\opt in \linkpatternnumbering
{
  \pgfmathsetmacro{\angle}{360/\lpsize*(\x-1)}
\ifx\xx\opt%
  \node[outer sep=1pt,anchor=180+\angle] at (\angle:1) {$\scriptstyle\xx$}; 
\else
  \node[outer sep=1pt,anchor=180+\angle,\opt] at (\angle:1) {$\scriptstyle\xx$}; 
\fi
}
\fi%
\iflinkpatterninverted%
\end{scope}
\fi%
\iflinkpatterntikzstarted\else%
\end{tikzpicture}%
\fi%
}}%
%
\newdimen{\loopcellsize}\setlength{\loopcellsize}{0.75cm}
\tikzset{bgplaq/.style={draw=black,fill=\linkpatternboxcolor}}
\def\plaqwest{}
\def\plaqeast{}
\def\plaqnorth{}
\def\plaqsouth{}
%
\def\plaqa{
\draw[/linkpattern/edge,\plaqwest,\plaqnorth] (0,0.5) .. controls (0,0.2) and (-0.2,0) .. (-0.5,0);
\draw[/linkpattern/edge,\plaqeast,\plaqsouth] (0,-0.5) .. controls (0,-0.2) and (0.2,0) .. (0.5,0);
}
\def\plaqb{
\draw[/linkpattern/edge,\plaqwest,\plaqsouth] (0,-0.5) .. controls (0,-0.2) and (-0.2,0) .. (-0.5,0);
\draw[/linkpattern/edge,\plaqeast,\plaqnorth] (0,0.5) .. controls (0,0.2) and (0.2,0) .. (0.5,0);
}
\def\plaqc{
\draw[/linkpattern/edge,\plaqsouth,\plaqnorth] (0,0.5) -- (0,-0.5);
\draw[/linkpattern/edge,\plaqwest,\plaqeast] (0.5,0) -- (-0.5,0);
}
\def\plaqj{
\draw[/linkpattern/edge,\plaqwest,\plaqnorth] (0,0.5) .. controls (0,0.2) and (-0.2,0) .. (-0.5,0);
}
\def\plaqr{
\draw[/linkpattern/edge,\plaqeast,\plaqsouth] (0,-0.5) .. controls (0,-0.2) and (0.2,0) .. (0.5,0);
}
\def\plaqh{
\draw[/linkpattern/edge,\plaqwest,\plaqnorth] (-0.5,0) -- (0.5,0);
}
\def\plaqv{
\draw[/linkpattern/edge,\plaqwest,\plaqnorth] (0,-0.5) -- (0,0.5);
}
\pgfkeys{/linkpattern/.cd,west/.store in=\plaqwest,east/.store in=\plaqeast,north/.store in=\plaqnorth,south/.store in=\plaqsouth}
\def\plaqname{plaq}
\newcommand\plaq[2][]{
\node[bgplaq,rectangle,draw,minimum size=\loopcellsize,transform shape] (\plaqname) {};
\useasboundingbox;
\pgfkeys{/linkpattern/.cd,#1}
\ifx#2\empty\else
\begin{scope}[x=\loopcellsize,y=\loopcellsize]
\csname plaq#2\endcsname
\end{scope}\fi
}
\newcommand\halfplaq[1][]{
\node[draw=none,fill=none,rectangle,minimum size=\loopcellsize,transform shape] (\plaqname) {};
\useasboundingbox;
\pgfkeys{/linkpattern/.cd,#1}
\begin{scope}[x=\loopcellsize,y=\loopcellsize]
\draw[bgplaq] (-0.5,-0.5) -- ++(1,1) -- ++(-1,0) -- cycle;
\draw[/linkpattern/edge,\plaqwest,\plaqnorth] (0,0.5) .. controls (0,0.2) and (-0.2,0) .. (-0.5,0);
\end{scope}
}
\tikzset{loop/.code={\def\plaqname{loop-\the\pgfmatrixcurrentrow-\the\pgfmatrixcurrentcolumn}},loop/.append style={matrix,row sep={\loopcellsize,between origins},column sep={\loopcellsize,between origins}}}
\def\linkpatternboxcolor{pink!50!white}

\section{Generic pipe dreams}\label{sec:intro}
Define a \defn{generic pipe dream tile} as any of the following:
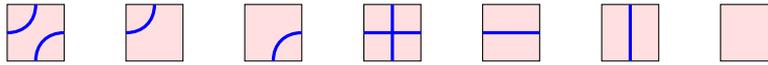
\begin{center}
\begin{tikzpicture}
\matrix[column sep=.8cm]{
\plaq{a} & \plaq{j} & \plaq{r} & \plaq{c} & \plaq{h} & \plaq{v} & \plaq{}
\\
\text{bump} \\
};
\end{tikzpicture}
\end{center}

Each pipe will carry a distinct label, generally from $[n] := \{1,\ldots,n\}$.
\junk{
  It will frequently be useful
  to consider tiles with one or zero pipes as bearing invisible pipes,
  i.e. ``invisible'' may be our $(n+1)$st label.
  For computational purposes, it is more efficient to disallow pipes with
  the same label to cross; for mathematical purposes, it is more natural
  to allow this but to sum over the two possibilities. 
  The blank edges may be fruitfully considered to bear pipes drawn in white,
  after which decision, in a fully blank tile one cannot be sure whether
  the two white pipes have crossed or not.
  Since the two cases
  are indistinguishable when both pipes are white, we will go with the
  ``disallow'' approach. \rem{why do we need this anyway?}
}

We assemble these tiles into $n\times n$ squares (and later, into
quadrangulated discs), calling them \defn{generic pipe dreams} or
\defn{GPDs}. For our first two applications of GPDs we insist that the
pipes down the West side are numbered $1\ldots n$ in order, that
the pipes across the North side are numbered $w^{-1}(1),\ldots,w^{-1}(n)$
for $w \in \mathcal S_n$, and that the East and South sides are blank.
The set of such GPDs is denoted $GPDs(w)$.
\junk{(The inverses could be avoided here but will be motivated
in \S \ref{ssec:GmodP}.)}
Here is one of the 45 GPDs for $w = 1243$:
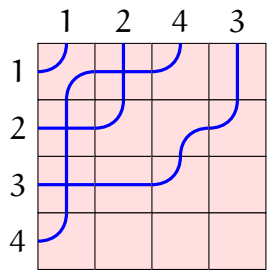
\begin{center}
  \begin{tikzpicture}\node[loop]{\plaq{a}&\plaq{c}&\plaq{j}&\plaq{v}\\\plaq{c}&\plaq{j}&\plaq{r}&\plaq{j}\\\plaq{c}&\plaq{h}&\plaq{j}&\plaq{}\\\plaq{j}&\plaq{}&\plaq{}&\plaq{}\\};
\foreach \x/\y in {1/1,2/2,3/4,4/3} {
  \path (loop-1-\x.north) node[above] {$\y$};
  \path (loop-\x-1.west) node[left] {$\x$};
  }
\end{tikzpicture}
\end{center}

\junk{
\begin{center}
\begin{tikzpicture}\node[loop] (l) {\plaq{c}&\plaq{a}&\plaq{j}\\\plaq{c}&\plaq{j}&\plaq{}\\\plaq{j}&\plaq{}&\plaq{}\\};
\foreach \x/\y in {1/3,2/1,3/2} {
  \path (loop-1-\x.north) node[above] {$\y$};
  \path (loop-\x-1.west) node[left] {$\x$};
}
\path (l.south) node[below] {$
  \begin{matrix}    
    (A+B)^4 \left(A+x_1-y_1\right) \left(A+x_2-y_1\right) \\
    \cdot \left(B-x_3+y_2\right) \left(B-x_2+y_3\right) \left(B-x_3+y_3\right)
  \end{matrix}$};
\begin{scope}[yshift=-4cm]
\node[loop] (l) {\plaq{c}&\plaq{j}&\plaq{v}\\\plaq{c}&\plaq{h}&\plaq{j}\\\plaq{j}&\plaq{}&\plaq{}\\};
\foreach \x/\y in {1/3,2/1,3/2} {
  \path (loop-1-\x.north) node[above] {$\y$};
  \path (loop-\x-1.west) node[left] {$\x$};
}
\path (l.south) node[below] {$
  \begin{matrix}
    (A+B)^3 \left(A+x_1-y_1\right) \left(A+x_2-y_1\right) 
    \left(A+x_2-y_2\right) \\
    \cdot \left(A+x_1-y_3\right) \left(B-x_3+y_2\right) \left(B-x_3+y_3\right)
  \end{matrix}
    $};
\end{scope}
\end{tikzpicture}
\end{center}
}

Define the \defn{generic pipe dream polynomial}
$G_w \in \ZZ[x_1,\ldots,x_n, y_1,\ldots,y_n, A, B]$
to be the sum $\sum_{\delta \in GPDs(w)} wt(\delta)$ over all GPDs $\delta$
for $w$, of a product of factors:
\begin{equation}\label{eq:wt}
wt(\delta) = \prod_{i,j \in [n]}
\begin{cases}
  A + x_i-y_j & \text{ if } \plaqctr{c}\ \plaqctr{v}\ \plaqctr{h} \text{ at }(i,j)\\[2mm]
  B -x_i+y_j & \text{ if } \plaqctr{} \text{ at }(i,j) \\[2mm]
  A+B  & \text{otherwise}
\end{cases}
\end{equation}
\junk{For example, the GPD above contributes\\
$(A+B)^7 \
\prod_{(i,j)=(3,4),(4,\ge 2)}(B-x_i+y_j) \
\prod_{(i,j)=(1,2),(1,4),(2,1),(3,1),(3,2)}(A+x_i-y_j)$.}

\smallskip
Let $w_0$ denote the longest element $\underline{n\ n-1\ldots 3\ 2\ 1} \in \mathcal S_n$.

\begin{Theorem}\label{thm:leadingTerms}
  \begin{enumerate}
  \item Consider the terms in $G_w$ with the highest power of $B$.
    The only GPDs that contribute to this $B$-leading form are the
    \defn{classic pipe dreams}, (i) with visible pipes in only the NW triangle
    and (ii) no two pipes crossing twice. The result is $B^{n^2-\ell(w)}$
    times the classic pipe dream formula for the double Schubert
    polynomial $S_w \in \ZZ[x_1,\ldots,x_n, y_1,\ldots,y_n]$,
    except evaluated at $S_w(A+x_1,\ldots,A+x_n, y_1,\ldots,y_n)$.
  \item Similarly, the only GPDs that contribute to the $A$-leading form are the
    \defn{bumpless pipe dreams} (albeit rotated $180^\circ$ from their
    definition in \cite{LLS-BPD}), with (i) no ``bump tiles'' \plaqctr{a}
    and (ii) no two pipes crossing twice. The result is $A^{n^2-\ell(w)}$
    times the bumpless pipe dream formula for $S_{w_0 w w_0}$,
    except evaluated at
    $S_{w_0w w_0}(B-x_n,\ldots, B-x_1,-y_n,\ldots,-y_1)$.
  \end{enumerate}
\end{Theorem}

\begin{proof}

  \begin{enumerate}
  \item If two pipes cross in tile $s$ and then again in tile $t$,
    we can replace each crossing tile with a bump tile,
    swapping the two pipe colors in the range between $s$ and $t$.
    This decreases the number of $A+x_i-y_j$ factors, allowing us to
    find another two factors of $B$.

    Now, if any pipe labelled $i \in [n]$ goes through a horizontal tile
    \plaqctr{h}, we can consider there to be an invisible
    pipe crossing it going North.
    The $i$-pipe heads East but eventually comes out the North side,
    whereas the invisible pipe heads North but comes out the East side,
    so the pipes must cross a second time. Apply the same argument as above.

    Hence the SE triangle must be solid invisible pipes, and the NW triangle
    full of visible pipes, no two crossing twice.
  \item The total length, in number of squares traversed, of pipe $i$
    is $i+w^{-1}(i)-1$; summing over $i$ we get
    ${n+1\choose 2}+{n+1\choose 2}-n = n^2$.
    Hence each tile used contains one visible pipe, on average.
    Some tiles accomodate two visible pipes, some zero, so
    \#\plaqctr c + \#\plaqctr a = \#\plaqctr {}.  We want to minimize the
    number of \plaqctr{}, as those don't admit an $A$ term, so we must
    minimize the number of crosses (to $\ell(w)$, using no double
    crosses) and bumps (to zero).  \qedhere
  \end{enumerate}
\end{proof}

\begin{Example}Here are the GPDs for $w=2431$:

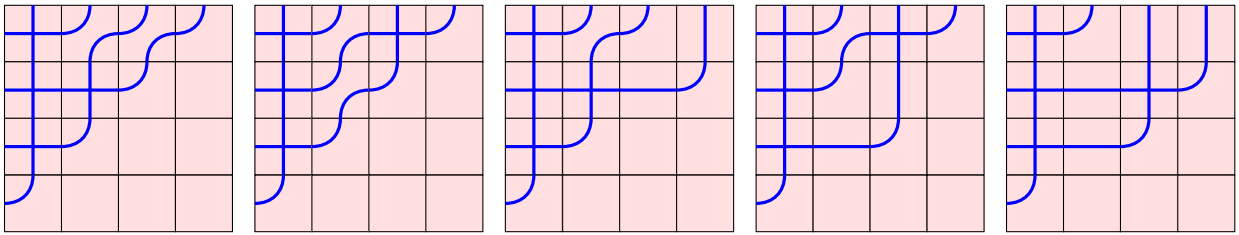
\begin{tikzpicture}\node[loop]{\plaq{c}&\plaq{a}&\plaq{a}&\plaq{j}\\\plaq{c}&\plaq{c}&\plaq{j}&\plaq{}\\\plaq{c}&\plaq{j}&\plaq{}&\plaq{}\\\plaq{j}&\plaq{}&\plaq{}&\plaq{}\\};\end{tikzpicture}\begin{tikzpicture}\node[loop]{\plaq{c}&\plaq{a}&\plaq{c}&\plaq{j}\\\plaq{c}&\plaq{a}&\plaq{j}&\plaq{}\\\plaq{c}&\plaq{j}&\plaq{}&\plaq{}\\\plaq{j}&\plaq{}&\plaq{}&\plaq{}\\};\end{tikzpicture}\begin{tikzpicture}\node[loop]{\plaq{c}&\plaq{a}&\plaq{j}&\plaq{v}\\\plaq{c}&\plaq{c}&\plaq{h}&\plaq{j}\\\plaq{c}&\plaq{j}&\plaq{}&\plaq{}\\\plaq{j}&\plaq{}&\plaq{}&\plaq{}\\};\end{tikzpicture}\begin{tikzpicture}\node[loop]{\plaq{c}&\plaq{a}&\plaq{c}&\plaq{j}\\\plaq{c}&\plaq{j}&\plaq{v}&\plaq{}\\\plaq{c}&\plaq{h}&\plaq{j}&\plaq{}\\\plaq{j}&\plaq{}&\plaq{}&\plaq{}\\};\end{tikzpicture}\begin{tikzpicture}\node[loop]{\plaq{c}&\plaq{j}&\plaq{v}&\plaq{v}\\\plaq{c}&\plaq{h}&\plaq{c}&\plaq{j}\\\plaq{c}&\plaq{h}&\plaq{j}&\plaq{}\\\plaq{j}&\plaq{}&\plaq{}&\plaq{}\\};\end{tikzpicture}

\junk{Summing up their weights, we find the generic pipe dream polynomial $G_w$:
\begin{multline*}
  \left(A^{5}+B^4(2A+x_1+x_2-y_2-y_3)+\cdots\right)
  \\
  \left(B -x_{3}+y_{4}\right)\left(B -x_{4}+y_{2}\right)\left(B -x_{4}+y_{3}\right)\left(B -x_{4}+y_{4}\right)
  \left(A +x_{1}-y_{1}\right)\left(A +x_{2}-y_{1}\right)\left(A +x_{3}-y_{1}\right)
\end{multline*}
where the $\cdots$ are terms of lower degree in $A$ or $B$.}

Only the first two GPDs are CPDs, and the $B$-leading form is related to
the double Schubert polynomial
$S_{2431} = (x_1-y_1)(x_2-y_1)(x_3-y_1)(x_1+x_2-y_2-y_3)$.

Similarly, only the last GPD is a BPD, and the $A$-leading form is related to
$S_{4213}=(x_1-y_1)(x_2-y_1)(x_1-y_2)(x_1-y_3)$.
\end{Example}

In the next sections we give two geometric applications
of GPD polynomials. 

\section{Lower-upper varieties}
We recall some constructions and results from \cite{Kn-uu}.
Let $B_-$, $B_+$ denote the groups of lower and upper triangular
invertible $n\times n$ matrices, respectively.
\junk{(Where $B_-$ is less
relevant we will abbreviate $B_+$ as $B$.) \rem{should we? there's already
  a $B$...}}
Define an action of
$B_- \times B_+$ on $(Mat_{n\times n})^2$ by
$ (b,c) \cdot (X,Y) := (b X c^{-1}, c Y b^{-1}) $.
\begin{Theorem}\cite{Kn-uu}
  Let $E \subseteq (Mat_{n\times n})^2$ denote the \defn{lower-upper scheme}
  $\{ (X,Y)\colon XY$ lower triangular, $YX$ upper triangular$\}$.
  Then $E$ is invariant under the above $(B_-\times B_+)$-action,
  and its components $(E_w\colon w\in \mathcal S_n)$ 
  are describable in two ways:
  \begin{eqnarray*}
    E_w &:=& \overline{ \{(X,Y) \in E\colon diag(XY) = w\cdot diag(YX)
               \text{ nonrepeating} \} } \\
  &=& \overline{(B_-\times B_+)\cdot \{(w,w^{-1} D)\colon D\text{ diagonal} \}}
  \end{eqnarray*}
\end{Theorem}

Hence the projection $(X,Y)\mapsto X$ of $E_w$ is the
\defn{matrix Schubert variety} $\barX_w := \overline{B_- w B_+}$,
abbreviated below as MSV,
whose $(B_-\times B_+)$-equivariant
cohomology class was shown in \cite{KM-Schubert} to be the double Schubert polynomial
$S_w(x_1,\ldots,x_n,y_1,\ldots,y_n)$.
The other projection $(X,Y)\mapsto Y$ has image $w_0 \barX_{w_0 w^{-1} w_0} w_0$.
\junk{whose class under the usual left-right action of $T\times T$ would be
$$ S_{w_0 w^{-1} w_0}(x_n,\ldots,x_1,y_n,\ldots,y_1),$$
as the outer $w_0$s reverse the alphabets. However, our action
of $B_-\times B_+$ is on the wrong sides (i.e. $(b,c)\cdot Y = c Y b^{-1}$)
so the correct polynomial is
$$ S_{w_0 w^{-1} w_0}(-y_n,\ldots,-y_1,-x_n,\ldots,-x_1). $$
Using the transpose symmetry, we can simplify this slightly to
$$ S_{w_0 w w_0}(x_n,\ldots,x_1,y_n,\ldots,y_1). $$
\rem{PZJ: I think the signs are missing in that last equation?
  AK: this is likely going to be tossed as too long, anyway}}

\begin{Theorem}\label{thm:GPDs}
  Let $(\CC^\times)^2$ act on $(Mat_{n\times n})^2$ by
  $(s,t)\cdot (X,Y) = (sX,tY)$, commuting with the $(B_-\times B_+)$-action,
  and write $\ZZ[A,B]$ for $H^*_{(\CC^\times)^2}(pt)$.
  The $(B_-\times B_+ \times (\CC^\times)^2)$-equivariant cohomology class of
  $E_w$ is $(A+B)^{-n}$ times the GPD polynomial $G_w$ from \S\ref{sec:intro}.
  (Since every pipe must turn East to North at some point, the overall
  $A+B$ exponent is nonnegative.)
\end{Theorem}
\begin{proof}[Sketch of proof]
  These cohomology classes are shown in \cite[\S 4]{artic39} to be uniquely determined by
  certain inductive divided difference formul\ae.
  It is not hard to show, using Yang--Baxter type arguments,
  that the generic pipe dream polynomials satisfy the same inductive
  formul\ae; see \cite[prop.~6]{artic81} for the $K$-theoretic version.
\end{proof}

In \cite[proposition 3]{artic39} we showed that one can compute the
equivariant classes of the projections of a $(\CC^\times)^2$-invariant
subvariety $Z \subseteq V \times W$ from the equivariant class of $Z$ itself,
using the $A$-leading and $B$-leading terms. Combining
theorem \ref{thm:GPDs} with the leading-term statements of
theorem \ref{thm:leadingTerms}, we recover the two standard formul\ae\
for double Schubert polynomials.

\begin{Corollary}\label{cor:GPDs}
  We compute the degree of $E_w$ by setting
  $x_\bullet = y_\bullet = 0$ to forget the $B_\pm$-actions, and $A = B = 1$
  for only the single scaling action.
  Hence $\deg E_w = \sum_{\delta\in GPDs(w)} 2^{\#\text{(tiles with turns)}-n}$.
\end{Corollary}

The proof of theorem~\ref{thm:GPDs} above does not explain where these GPD formul\ae\ come from (and in fact,
it is not how we found them). Similarly to the case of MSVs treated in \cite{KM-Schubert},
one can interpret them in terms of a Gr\"obner degeneration.
This has the added benefit that it gives us insight on ``how'' pipe dreams
{\em know about their connectivity,}
something which is not readily available in the study of MSVs.

\begin{Theorem}
There is an equivariant degeneration of $E_w$ to a union
  of quadratic complete intersections $F_\delta$ (possibly with some embedded
  components not affecting the $H^*$-class), one for each GPD $\delta$,
  whose classes are the individual terms in the formula for $(A+B)^{-n}G_w$.
\end{Theorem}
\begin{proof}[Sketch of proof]
  Write $t_i=(XY)_{ii}$, so that $t_{w^{-1}(i)}=(YX)_{i,i}$ in $E_w$.

  Introduce \defn{flux} variables on edges of the $n\times n$ square lattice by
\[
\Phi_e = \sum_{\substack{\text{squares $(i,j)$}\\\text{right of $e$ if $e$ vertical}\\\text{below $e$ if $e$ horizontal}}} X_{ij}Y_{ji}
\]
In particular, note that in $E_w$, the fluxes on boundary edges are fixed to be
zero on the South and East sides, $t_1,\ldots,t_n$ on the west side, $t_{w^{-1}(1)},\ldots,t_{w^{-1}(n)}$ on the north side, which matches the connectivity of GPDs associated to $E_w$.
Also, they satisfy the conservation equation at each square
\begin{equation}\label{eq:flux}
\begin{tikzpicture}[baseline=-3pt]
  {\setlength{\loopcellsize}{1.25cm}\plaq{}}
\node[shape=isosceles triangle,shape border rotate=90,inner sep=2pt,fill=blue,label={above:$\Phi_{\text{N}}$}] at (plaq.north) {};
\node[shape=isosceles triangle,shape border rotate=90,inner sep=2pt,fill=blue,label={below:$\Phi_{\text{S}}$}] at (plaq.south) {};
\node[shape=isosceles triangle,shape border rotate=0,inner sep=2pt,fill=blue,label={right:$\Phi_{\text{E}}$}] at (plaq.east) {};
\node[shape=isosceles triangle,shape border rotate=0,inner sep=2pt,fill=blue,label={left:$\Phi_{\text{W}}$}] at (plaq.west) {};
\end{tikzpicture}
\qquad
\Phi_W+\Phi_S=\Phi_E+\Phi_N
\end{equation}

Now consider the following degeneration of the whole scheme $E$:
start from its defining ideal $\mathcal I=\left<(XY)_>,(YX)_<\right>$,
give a weight
\[
  [X_{ij}]=-ij\qquad [Y_{ij}]=ij\qquad i,j=1,\ldots,n
\]
to each variable, and take the initial ideal $\init(\mathcal I)$ with
respect to the corresponding monomial
order.~\footnote{If we view $(Mat_{n\times n})^2$ as
  $T^*Mat_{n\times n}$, then this degeneration preserves the symplectic form.}
Note that this is only a {\em partial}\/ Gr\"obner degeneration,
because ties remain among monomials, which means that
$\init(\mathcal I)$ need not be a monomial ideal.  In particular,
fluxes, and \eqref{eq:flux}, are unaffected by the degeneration.

We need to compute some of $\init(\mathcal I)$. Consider the
``overlap'' $(XY)_> X - X (YX)_<$.
An easy calculation of its initial term leads to the equation
\[
  X_{ij}(\Phi_S-\Phi_E)=0\qquad i,j=1,\ldots,n
\]
for the degeneration of $E$,
where $\Phi_S$ and $\Phi_E$ are as in \eqref{eq:flux} at the square $(i,j)$.

Note that if $X_{ij}=0$, then $\Phi_S=\Phi_N$ and $\Phi_W=\Phi_E$. We conclude that
in each component of the degeneration of $E$, fluxes either propagate horizontal/vertically
or diagonally
(i.e., the equation \eqref{eq:flux} splits into two cases, with no more mixing of fluxes).
Now starting from the bottom/left and adding one square at a time, we conclude that
fluxes $\Phi_e$ can only take the values $0,t_1,\ldots,t_n$. Draw a pipe (labelled $i$)
across each edge such that $\Phi_e=t_i$, and leave the zero flux edges blank: one obtains
this way a GPD.

With a bit more work\rem{ahem}, one concludes that the degeneration of $E$ is contained
inside the union of $F_\delta$ over $\delta$ GPD, where
    \[
  F_\delta = \left\{ (X,Y)
\in Mat_{n\times n}^2:\ 
\begin{aligned}
\text{At each square }(i,j),\ X_{ij}&=0\text{ if $(i,j)$ is a crossing/straight pipe}
\\
Y_{ji}&=0\text{ if $(i,j)$ is blank}
\\
\text{At each edge }e,\ \Phi_e&=\begin{cases}t_i&\text{if pipe $i$ passes through $e$}\\0&\text{else}\end{cases}
\end{aligned}
\right\}
\]
up to lower-dimensional (necessarily embedded) components.
If we degenerate $E_w$ instead, this amounts to imposing the outgoing fluxes, i.e.,
that the connectivity of the GPDs be $w$. 

Finally, one checks that $F_\delta$ is a complete intersection of dimension
$n(n-1)$: each square provides one equation, but by easy linear algebra, among the flux equations, $n$ of them are redundant.
The cohomology class of $F_\delta$ can then be computed by taking the product of the weights of its equations, i.e.,
$wt(X_{ij})=A+x_i-y_j$, 
$wt(Y_{ij})=B-x_i+y_j$,
$wt(\Phi_e)=A+B$,
to be compared with \eqref{eq:wt}.

Because we already know equality of cohomology
classes according to theorem~\ref{thm:GPDs}, we conclude that the degeneration of $E_w$
can differ from the union of corresponding $F_\delta$ by at most lower-dimensional
embedded components.
\end{proof}

\subsection{Degree of the commuting variety}
The lower-upper scheme was invented in \cite{Kn-uu} to study the
\defn{commuting scheme} $C := \{(X,Y) \in (Mat_{n\times n})^2\colon XY = YX\}$.
Specifically, the first author showed that $C$ has a
degeneration to the lower-upper variety $E_1$ union (possibly)
some embedded components; ergo, $\deg C = \deg E_1$.
With corollary~\ref{cor:GPDs}, we can compute that as a $2$-enumeration:
\begin{Example}
  When $n=3$, $\deg C = 1 + 2 + 2 + 2 + 4 + 4 + 8 + 8 = 31$:\\
  {\setlength{\loopcellsize}{0.6cm}
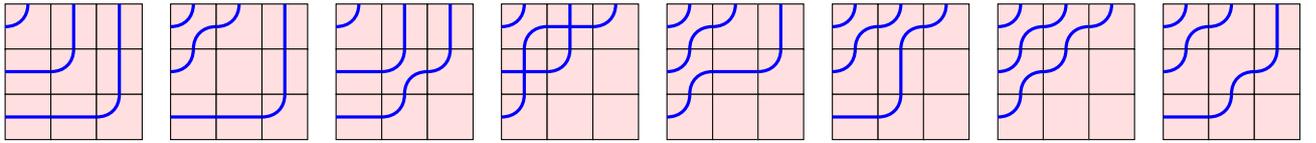
\begin{tikzpicture}
\node[loop]{\plaq{j}&\plaq{v}&\plaq{v}\\\plaq{h}&\plaq{j}&\plaq{v}\\\plaq{h}&\plaq{h}&\plaq{j}\\};\end{tikzpicture} 
\begin{tikzpicture}
\node[loop]{\plaq{a}&\plaq{j}&\plaq{v}\\\plaq{j}&\plaq{}&\plaq{v}\\\plaq{h}&\plaq{h}&\plaq{j}\\};\end{tikzpicture} 
\begin{tikzpicture}
\node[loop]{\plaq{j}&\plaq{v}&\plaq{v}\\\plaq{h}&\plaq{a}&\plaq{j}\\\plaq{h}&\plaq{j}&\plaq{}\\};\end{tikzpicture} 
\begin{tikzpicture}
\node[loop]{\plaq{a}&\plaq{c}&\plaq{j}\\\plaq{c}&\plaq{j}&\plaq{}\\\plaq{j}&\plaq{}&\plaq{}\\};\end{tikzpicture} 
\begin{tikzpicture}
\node[loop]{\plaq{a}&\plaq{j}&\plaq{v}\\\plaq{a}&\plaq{h}&\plaq{j}\\\plaq{j}&\plaq{}&\plaq{}\\};\end{tikzpicture} 
\begin{tikzpicture}
\node[loop]{\plaq{a}&\plaq{a}&\plaq{j}\\\plaq{j}&\plaq{v}&\plaq{}\\\plaq{h}&\plaq{j}&\plaq{}\\};\end{tikzpicture} 
\begin{tikzpicture}
\node[loop]{\plaq{a}&\plaq{a}&\plaq{j}\\\plaq{a}&\plaq{j}&\plaq{}\\\plaq{j}&\plaq{}&\plaq{}\\};\end{tikzpicture} 
\begin{tikzpicture}
\node[loop]{\plaq{a}&\plaq{j}&\plaq{v}\\\plaq{j}&\plaq{r}&\plaq{j}\\\plaq{h}&\plaq{j}&\plaq{}\\};\end{tikzpicture} 
}
\end{Example}
This formula is quite computationally effective,
allowing us to go up to
\[
\deg C_{16}=8\,152\,788\,880\,952\,641\,347\,488\,179\,079\,698\,833\,772\,730\,621\,821\,001\,288\,826\,319\,965\,501\,665  
\]

See also \cite[theorem~4]{artic81} for
an independent proof of this formula for the (multi)degree of the
commuting variety.

\section{Schwartz--MacPherson classes of
  Kazhdan--Lusztig varieties and of some unions thereof}\label{sec:SM}
\newcommand\fl{G/B_+}
\subsection{Segre--Schwartz--MacPherson classes on $\fl$}
Let $X \subseteq Y$ be a locally closed subscheme in
a smooth ambient variety, over $\CC$. To this one associates a
\defn{Chern--Schwartz--MacPherson class} $\csm(X)$ in
$H_*^{BM}(Y) \iso H^{\dim_\RR Y-*}(Y)$,
where $H_*^{BM}$ denotes the Borel--Moore homology.
These classes are characterized by three properties:
\begin{enumerate}
\item If $X = X_1 \sqcup X_2$ is a disjoint union,
  then $\csm(X) = \csm(X_1) + \csm(X_2)$.
\item If $f:Y\to Z$ makes $X$ a bundle over $f(X)$ with fiber $F$, then
  $f_*(\csm(X)) = \csm(f(X))\, \chi_c(F)$, where $\chi_c$ is the compactly
  supported Euler characteristic.
\item If $X=Y$ is proper, then $\csm(X) = c(TY)$, 
  the total Chern class.
\end{enumerate}

\junk{In the original works on the subject, (e.g. MacPherson's proof of the Deligne--Grothendieck conjecture, that CSM classes exist) 
the classes defined are inhomogeneous, with $\hbar \mapsto -1$, taking $e(TY)$ to the total Chern class.
This loses no information but is frightfully unnatural, insofar as when one
considers cohomology as the associated graded of $K$-theory,
it {\em shouldn't contain} inhomogeneous elements.
\rem{might lose this unnecessary paragraph}}

These CSM classes behave well under pushforward;
for good Poincar\'e-duality properties, 
one defines the
\defn{Segre--Schwartz--MacPherson class} $\ssm(X\subseteq Y) := \csm(X)/c(TY)$,
which lives properly in a certain localization of $H^*(Y)$.
Here are two naturality properties these SSM classes possess:

\begin{Lemma}\label{lem:ssm}
  \begin{enumerate}
  \item Under the isomorphism $H^*(Y) \iso H^*(Y\times V)$,
    for $V$ a vector space, we have $\ssm(X) \mapsto \ssm(X \times V)$.
\junk{
    (There can of course be no similar statement for CSM classes, for degree
    reasons; however SSM classes all have formal degree $0$.)
    \rem{optional fact}}
\item \cite{Schurmann} Let $X_1,X_2 \subseteq Y$ be locally closed
  submanifolds,
  whose closures $\overline X_1,\overline X_2$ we stratify by manifolds
  (having $X_1,X_2$ as strata).
    If each stratum in $\overline X_1$ is transverse to each in
    $\overline X_2$, then
    $\ssm(X_1\cap X_2 \subseteq Y)$ is the product
    $    \ssm(X_1 \subseteq Y)\,    \ssm(X_2 \subseteq Y)$. 
  \end{enumerate}
\end{Lemma}

All of the foregoing generalizes nicely to equivariant cohomology,
i.e.  if $T$ acts on $Y$ preserving $X$ then $\csm(X)$
can be defined in $H^*_{T}(Y)$. We adapt 
\cite[lemma~2]{GK-Schubert} to compute point restrictions of SSM classes
on the flag variety $\fl$, where $G$ is a semisimple algebraic group and $B_+$ is its Borel subgroup, 
giving an alternate proof of \cite[theorem 1.1]{CSu}.

Let $W$ be the Weyl group of $G$.
\begin{Lemma}\label{lem:pairings}
  Let $X^v_\circ := B_+vB_+/B_+$, $X^\circ_w := B_- wB_+/B_+$ be the Bruhat and
  opposite Bruhat cells in $\fl$. Then the bases $(\csm(X^v_\circ))_{v\in W}$,
  $(\ssm(X_w^\circ))_{w\in W}$ are dual bases under the Poincar\'e pairing
  $\langle \alpha,\beta \rangle := \int \alpha \beta$
  on $H^*_{T}(\fl)$ and localizations thereof.
\end{Lemma}

\begin{proof}
  Let $\pi\colon \fl\to pt$ be the map to a point, so $\pi_*$ is
  integration over $\fl$. Then we compute the Poincar\'e pairings:
  \begin{alignat*}{3}
    \int_{\fl} \csm(X^v_\circ) \ssm(X_w^\circ) 
    &= \int_{\fl} \csm(X^v_\circ \cap X_w^\circ)
         && \text{\quad by lemma \ref{lem:ssm}(2)}\\
    &= \chi_c( X^v_\circ \cap X_w^\circ ) 
         && \text{\quad by property (2) of CSM classes} \\
    &= \chi_c( (X^v_\circ \cap X_w^\circ)^T ) 
         && \text{\quad a property of $\chi_c$} \\
    &= \chi_c( (X^v_\circ)^T \cap (X_w^\circ)^T )&
    =&\ \chi_c( \{vB_+/B_+\} \cap \{wB_+/B_+\} ) 
    = \delta_{vw} \quad \qedhere
  \end{alignat*}
\end{proof}

We will make use of a similar pair of cell decompositions of the
{\em degenerate Bott--Samelson variety} from \cite{KaruppuchamyParameswaran}.
For $Q$ a word (not necessarily reduced)
in the simple reflections of $G$'s Weyl group,
one associates a \defn{Bott--Samelson manifold} $BS^Q$ and $B_+$-equivariant map
$BS^Q \to \fl$. $BS^Q$ can be defined in terms of the \defn{heap} of $Q$, which
in type $A$ can be drawn as the dual quadrangulation of the wiring diagram of $Q$, e.g.,
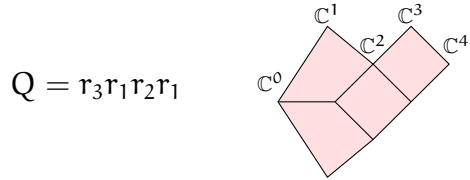
\[
Q=r_3r_1r_2r_1\qquad \begin{tikzpicture}[x={(.5cm,0cm)},y={(0cm,-.5cm)},baseline={([yshift=-\the\dimexpr\fontdimen22\textfont2\relax]current  bounding  box.center)}]
\begin{scope}[every path/.style={draw}]
\path[fill=\linkpatternboxcolor] (-.2,0) node[above=-1mm] {$\ss \CC^1$} -- (1,1) node[above] {$\ss \CC^2$} -- (2,-0) node[above=-1mm] {$\ss \CC^3$} -- (3,1) node[above,xshift=1mm] {$\ss \CC^4$} -- (2,2) -- (1,3) -- (-.2,4) -- (-1.5,2) node[above,xshift=-1mm] {$\ss \CC^0$} -- cycle;
\path (1,1) -- (0,2);
\path (0,2) -- (-1.5,2);
\path (0,2) -- (1,3);
\path (1,1) -- (2,2);
\end{scope}
\end{tikzpicture}
\]
where vector subspaces of $\CC^n$
sit at every vertex of the diagram, satisfy left-to-right inclusion, and the vertices at the top of the picture
are fixed to form the standard flag $\CC^0\subset\CC^1\subset\cdots\subset\CC^n$. The map to the complete flag variety
$\fl$ is reading off the bottom flag.

This $BS^Q$ has a nice $B_+$-invariant cell decomposition
$BS^Q = \coprod_{R\subseteq Q} BS^R_\circ$ indexed by the $2^Q$ ``subwords'' of $Q$,
which also label the $T$-fixed points $(BS^R)^T$;
the cell $BS^R_\circ$ corresponds to imposing equality (resp.\ inequality) of top and bottom
subspaces of a square corresponding to a letter of $R$ (resp.\ of $Q\backslash R$).

\newcommand\BSd{TV}

There is a fascinating toric degeneration $BS^Q \rightsquigarrow \BSd^Q$
in which each stratum $BS^R := \overline{BS^R_\circ}$ degenerates to
a toric subvariety $\BSd^R$ (i.e., it stays irreducible).
This toric variety $\BSd^Q$ is modeled on a combinatorial cube $\square^Q$
\cite{HaradaYang} with vertices indexed by $2^Q$,
and the $\BSd^R$ correspond to the faces of $\square^Q$ containing
the bottom vertex (the empty subword).
As such $\BSd^Q$ has a second cell decomposition
$\BSd^Q = \coprod_{R\subseteq Q} \BSd_R^\circ$ transverse to the first,
whose closures correspond to
the faces containing the top vertex $R=Q$. These two cell decompositions
then enjoy the same dual-basis property as in lemma \ref{lem:pairings}.

This toric degeneration has the uncommon property of being smooth, and
consequently, there is a $T_c$-equivariant diffeomorphism $\BSd^Q \to BS^Q$
where $T_c$ is the maximal compact subgroup of the complex torus $T$.
One can carry the algebraic submanifolds
$\overline{\BSd_R^\circ} \subseteq \BSd^Q$ across this diffeomorphism
to give {\em non}algebraic submanifolds of $BS^Q$;
see \cite[chapter 18]{AndersonFulton}.
We involve this variety $TV^Q$
so as to use the same calculation as in lemma \ref{lem:pairings}.

\begin{Theorem}[see also \cite{CSu}]\label{thm:CSu}
  Let $Q$ be a word in the simple reflections of $W$, with product $v\in W$.
  Then the restriction $\ssm(X_w^\circ)|_v$ of the class $\ssm(X_w^\circ)$
  to the point $vB_+/B_+$ is given by the sum over all subwords $R\subseteq Q$
  with product $w$, of
  \begin{equation}\label{eq:main}
    \prod_{i\in I} \frac{\beta_i}{\beta_i+1} \prod_{i\not\in I} \frac{1}{\beta_i+1}\qquad \beta_i = \left(\prod_{j<i} r_{Q_j}\right) \cdot \alpha_{Q_i}\qquad R=\prod_{i\in I} Q_i
  \end{equation}
In type $A$, we can index these terms using pipe dreams made of \plaqctr{a} and \plaqctr{c}
in the heap of $Q$,
where the simple roots are parameterised as $\alpha_{r_i}=x_{i}-x_{i+1}$, and the roots $\beta_i$ can
be read off the diagram as the differences of labels propagating on parallel sides of squares,\\
 e.g., if $Q=r_3 r_4 r_2 r_1 r_2 r_3$, $R=r_3\_\,\_r_1r_2\_$,
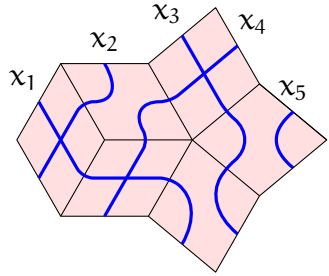
\begin{tikzpicture}[execute at begin picture={\clip (-1.4239,-1.8845) rectangle ++(4.3798,3.7691);},x={(1.2821cm,0cm)},y={(0cm,-1.2821cm)},baseline={([yshift=-\the\dimexpr\fontdimen22\textfont2\relax]current  bounding  box.center)},line join=round,scale=.9]
\begin{scope}[every path/.append style={draw=black}]
\path[fill=\linkpatternboxcolor] (-.5,.877) -- (.5,.877) -- (1.266,1.5198) -- (1.766,.64279) -- (2.5321,-0) -- (1.766,-.64279) -- (1.266,-1.5198) -- (.5,-.877) -- (-.5,-.877) -- (-1,-0) -- cycle;
\path (1,-0) -- (1.766,-.64279);
\path (1,-0) -- (0,-0);
\path (-.5,-.877) -- (0,-0);
\path (-.5,.877) -- (0,-0);
\path (1,-0) -- (1.766,.64279);
\path (1.766,-.64279) -- (2.5321,-0);
\path (1,-0) -- (.5,-.877);
\path (1,-0) -- (.5,.877);
\end{scope}
\begin{scope}[every path/.append style={/linkpattern/edge,rounded corners}]
\path svg[xscale=1.2821cm,yscale=-1.2821cm] {M -.75 .4385 L -.25 -.4385 Q .25 -.4385 0 -.877} node[above] {$x_2$};
\path svg[xscale=1.2821cm,yscale=-1.2821cm] {M 0 .877 L .5 -0 Q .25 -.4385 .75 -.4385 L 1.516 -1.0813} node[above,xshift=2mm] {$x_4$};
\path svg[xscale=1.2821cm,yscale=-1.2821cm] {M .88302 1.1984 Q 1.133 .75989 .75 .4385 L -.25 .4385 L -.75 -.4385} node[above,xshift=-2mm] {$x_1$};
\path svg[xscale=1.2821cm,yscale=-1.2821cm] {M 1.516 1.0813 Q 1.133 .75989 1.383 .32139 Q 1.766 -0 1.383 -.32139 L .88302 -1.1984} node[above,xshift=-2mm] {$x_3$};
\path svg[xscale=1.2821cm,yscale=-1.2821cm] {M 2.1491 .32139 Q 1.766 -0 2.1491 -.32139} node[above] {$x_5$};
\end{scope}
\end{tikzpicture}
\quad
$\begin{aligned}\beta_1&=x_3-x_4\\\beta_2&=x_3-x_5\\\beta_3&=x_2-x_4\\\beta_4&=x_1-x_4\\\beta_5&=x_1-x_2\\\beta_6&=x_1-x_5
\end{aligned}$
\end{Theorem}

\begin{proof}
  We compose $\{Q\} \into \BSd^Q \xrightarrow\sim BS^Q \to \fl$.
  The image of the point $Q$ in $\fl$ is $vB_+/B_+$ by assumption.
  We first compute the pushforward map of the composite $\BSd^Q \to \fl$
  in one pair of bases, then transpose it to compute pullback in
  the dual bases.

  Push forward $\csm(\BSd^R_\circ) \mapsto \csm(BS^R_\circ)$ along the
  middle map, then consider the $B_+$-equivariant map $BS^R_\circ \to \fl$.
  Its image is some union $\union X^w_\circ$ of $B_+$-orbits, and 
  the map must be a bundle over each target orbit.
  (In fact a trivial bundle, as each $X^v_\circ$ is a {\em free} orbit under
  the subgroup $B_+ \cap (v [B_-,B_-] v^{-1})$).
  So far we know $\csm(BS^R_\circ) \mapsto \sum_w \chi_c(F_w) \csm(X^w_\circ)$,
  where $F_w$ is the fiber over $wB_+/B_+$. Now use
  $$ \chi_c(F_w) = \chi_c((F_w)^T)
  = \chi_c(F_w \cap (BS^R_\circ)^T)
  = \chi_c(F_w \cap \{R\}) = [R \in F_w] = \left[\,\prod R = w\,\right] $$
  where $[$assertion$] = 1$ if true, $0$ if false.
  In all, $\csm(\BSd^R_\circ) \mapsto \csm(X^{\prod R}_\circ)$.

  Transposing, $\ssm(X_w^\circ) \mapsto \sum \{ \ssm(\BSd_R^\circ)\colon
  R\subseteq Q,\ \prod R = w\}$.
  Consequently,
  $$ \ssm(X_w^\circ)|_v
  = \sum \left\{ \ssm(\BSd_R^\circ)|_Q \colon
    R\subseteq Q,\ \prod R = w \right\}
  $$
  Note that $\BSd^\circ_R = \bigcap_{r\in R} \BSd_r \cap\bigcap_{r\in Q\backslash R} (\BSd^Q\backslash \BSd_r) $,
  and this intersection is transverse in the sense of lemma \ref{lem:ssm}(2). 
To calculate the factor $\ssm(\BSd_R^\circ)|_Q$, we work in
  the open set $\BSd^Q_\circ \iso \CC^Q$, whose intersection with $\BSd_r$
  corresponds to the subset
  $\{ \vec v \in \CC^Q\colon v_r = 0\}$.
  \eqref{eq:main} then follows from lemma~\ref{lem:ssm}.
\end{proof}

It is a standard fact that the pullback of a {\em Schubert} class
$[\overline{B_- w P_+}/P_+] \in H^*(G/P_+)$ from a partial flag manifold,
along the projection $\fl \onto G/P_+$ ($P_+\supseteq B_+$),
is again a Schubert class
$[\overline{B_- w B_+}/B_+] \in H^*(\fl)$, where $w$ is the unique
smallest element in its coset $w W_P$.
The situation is more complicated for CSM and SSM
classes. The argument used in theorem \ref{thm:CSu} lets one show
$\pi_*(\csm(X^{v'}_\circ)) = \csm(X^v_\circ)$ for any $v' \in v W_P$,
so transposing, $\pi^*(\ssm(X_v^\circ)) = \sum_{f \in W_P} \ssm(X_{vf}^\circ)$.
Thus to compute the point restriction of a $G/P_+$ SSM class 
we can sum the pipe dreams over all $vW_P$.
\junk{, but there is a more efficient approach, which for concreteness we describe in type $A$.

\begin{Theorem}
  Regard $v W_P \in \mathcal S_n/W_P$ as a permutation with ambiguous positions,
  e.g. if $W_P = \mathcal S_k \times \mathcal S_{n-k}$ then the first $k$ values are
  in a jumbled-up bag, likewise the last $n-k$.
  Then $(w W_P)^{-1}$ has ambiguous {\em values}, hence can be encoded
  as a string $\sigma$ (in some ordered alphabet) with repeats, e.g. a
  binary string with content $0^k 1^{n-k}$.

  If $Q$ is a word for some permutation
  in $v W_P$, we can compute $\ssm(X_{w W_P}^\circ)|_{v W_P}$ as a sum
  over GPDs in $Q$'s heap, one end \rem{pick conventions} labeled by
  $\sigma$ in sorted order, the other by $\sigma$.
\end{Theorem}

\begin{proof}
  Let $\Delta$ be the set of theorem \ref{thm:CSu}'s $Q$-shaped GPDs
  for calculating $\sum_{f \in W_P} \ssm(X_{vf}^\circ)|_w$.
  In each $\delta \in \Delta$, consider the set $S(\delta)$ of squares
  whose two pipes have $W_P$-equivalent labels. Group $\Delta$ into
  equivalence classes $\Delta_S$ based on this set $S(\delta)$.
  Then the claim is that $\Delta_S$ has $2^{\#S}$ elements, and the sum of
  its GPDs' weights is computed by a single $Q$-shaped GPD with repeated labels.
\end{proof}
}
This gives us combinatorial formula\ae\ for SSM classes in $G/P_+$, which are spelled out in
\cite[lemma 2.4 and \S5]{artic80}.

\smallskip
Here we focus on a different direction, which will allow us to reconnect to GPDs.

\subsection{Open Kazhdan--Lusztig and matrix Schubert varieties}\label{ssec:KL}
Given a word $Q$ for $v\in\mathcal \mathcal S_n$,
define a modified weight for a GPD $\delta$
on the heap of $Q$:
\begin{equation}\label{eq:wt2}
\widetilde{wt}(\delta) = \prod_{\setlength{\loopcellsize}{0.4cm}\tikz[baseline=0pt]{\path (0,0.5);\plaq{}\path (plaq.east) --++(.1,0) node {$\ss x$};\path (plaq.north) --++(0,.15) node {$\ss y$};}}
\begin{cases}
  x-y & \text{ if } \plaqctr{c}\ \plaqctr{v}\ \plaqctr{h}\\[2mm]
  x-y+1 & \text{ if } \plaqctr{}\\[2mm]
  1  & \text{otherwise}
\end{cases}
\end{equation}
It differs from the original weight \eqref{eq:wt} by signs and the specialization $A=0$, $B=-1$.~\footnote{
We could have kept $B$ unspecialized, which would correspond to the natural homogenization of CSM classes in
$H^{\dim Y}(Y)[B] \iso H^{\dim Y}_{\CC^\times}(Y)
\iso H^{\dim Y}_{\CC^\times}(T^* Y) $,
where $B$ is interpreted as equivariant parameter for the scaling of the fiber of $T^*Y$.}

Given a set of GDPs, define the modified GPD polynomial to be the sum of its modified weights.
(If locations of endpoints and connectivity of pipes are fixed, then GPD and modified GPD polynomials only differ by an overall sign
and the specialization above.)

An equivalent formulation of theorem~\ref{thm:CSu} is:
\begin{Corollary}\label{cor:KL}
  The SSM class of the \defn{open Kazhdan--Lusztig variety}
  $X_w^\circ \cap X^v_\circ$ inside the cell $X^v_\circ$ is computed
  by the same formula \eqref{eq:main} as in theorem \ref{thm:CSu}. In type $A$, and taking the word $Q$ of $v$ to be reduced,
  the CSM class of $X_w^\circ \cap X^v_\circ$ inside $X^v_\circ$
  is given by the modified GPD polynomial 
  for the set of pipe dreams on the heap of $Q$ in theorem \ref{thm:CSu}.
\end{Corollary}

\begin{proof}
  Let $U_v := vB_- B_+/B_+$ be the big cell in $\fl$ centered at the point $vB_+/B_+$. 
  The Kazhdan--Lusztig lemma 
  is a $T$-equivariant isomorphism of pairs
  $$ (X_w^\circ \cap U_v \ \subseteq\ U_v)
  \iso (X_w^\circ \cap X^v_\circ \ \subseteq\ X^v_\circ) \times X_v^\circ $$
  We then apply lemma \ref{lem:ssm}(1).

  To compute the CSM class we must multiply by the total Chern class $c(T X^v_\circ)$,
  which is nothing but $\prod_{i=1}^{|Q|}(1+\beta_i)$. \rem{justify, this is where we use $Q$ reduced}
  This removes the denominator in the expression of \eqref{eq:main}, so that
  now a \plaqctr{a} contributes $1$, whereas a \plaqctr{c} contributes
  $\beta_i=x-y$ where $x$ and $y$ are labels attached to the two sides of the square.
  This matches the weights of \eqref{eq:wt2} (noting that these GPDs have no blanks).
\end{proof}

One well-studied family of such varieties arises in the theory of
cluster algebras:
\begin{Proposition}\label{prop:dBru}
  Let $(B_+ u B_+) \cap (B_- v B_-) \subseteq GL_n$
  be $w_0$ times the \defn{double Bruhat cell}.
  Its $(T^n\times T^n)$-equivariant CSM class can be computed as
  the modified GPD polynomial for pipe dreams on a $n\times n$ square,
  made of \plaqctr{a} and \plaqctr{c},
  with boundary $1\ldots 2n$ down the West
  then along the South side, $u^{-1}(1)\ldots u^{-1}(n)$
  across the North, and $n+v^{-1}(1)\ldots n+v^{-1}(n)$
  down the East.  
\end{Proposition}

\begin{proof}
  Let $u\oplus v \in \mathcal S_{2n}$ have one-line notation
  $\underline{u(1)\ldots u(n)\ n+v(1)\ldots n+v(n)}$.
  Its Fulton essential set is $ess(u\oplus v) =
  \begin{bmatrix}     ess(u) \\ & ess(v)  \end{bmatrix}$.
  The pullback of $X_{u \oplus v}^\circ$ along the Fulton isomorphism
  \cite{Ful-flags}
  $
  Mat_{n\times n} \xrightarrow\sim X^{\underline{n+1\ldots 2n\ 1\ldots n}}_\circ
  \subseteq GL(2n)/B_+$, $  M \mapsto \begin{bmatrix} M&I_n\\ I_n&0 \end{bmatrix}$
  is the double Bruhat cell. The cut-the-deck permutation
  $\underline{n+1\ldots 2n\ 1\ldots n}$ is 321-avoiding,
  hence is ``fully commutative'' i.e. has only one reduced word up
  to commuting moves. The unique resulting heap is the $n\times n$ square.
  Now apply corollary \ref{cor:KL}.
\end{proof}

Our final application is the following:
\begin{Theorem}
  Let $w$ be a \defn{partial permutation}, viewed as a matrix in $Mat_{k\times n}$.
  Then the CSM class of $B_-w B_+$ is the modified GPD polynomial of GPDs on a $k\times n$ rectangle,
  such that pipes come in from the West side,
  and the $i^{\rm th}$ pipe on the West side (counted from bottom to top) emerges on the North side
  at location $j$ if $w_{ij}=1$, anywhere on the East side otherwise.
\end{Theorem}
\begin{proof}
Consider $v=\underline{n+1\ldots n+k\,1\ldots n}$. The setup is similar to that of
proposition~\ref{prop:dBru}, and
it is clear that via the Fulton isomorphism, \rem{is it?}
\[
  B_-w B_+ = \bigsqcup_{\substack{w'\in\mathcal \mathcal S_{n+k}\\ w'|_{k\times n}=w}} X^\circ_{w'}\cap X^v_o
\]
Therefore, $\csm(B_-w B_+)=\sum_{w'|_{k\times n}=w} \csm(X^\circ_{w'}\cap X^v_o)$,
with each of the summands being described
according to corollary~\ref{cor:KL} by pipe dreams in a $k\times n$ rectangle made of
\plaqctr{a} and \plaqctr{c} such that the pipes coming from the West
corresponding to rows of $1$s of $w$ come out North
as prescribed by $w$, whereas the other pipes coming from the West or South come
out North or East as prescribed
by $w'$. Summing over $w'$ removes this last condition. Now {\em erase}\/ all the pipes coming from the South,
noting that the only information that's lost is when two such paths bump into or cross
each other; but according to \eqref{eq:wt2},
the weight of a blank is the sum of the weight of a bump and of that of a cross.
\end{proof}

\begin{Example}
  If $w=\begin{pmatrix}1&0\\0&0
  \end{pmatrix}$, there are three GPDs:
  
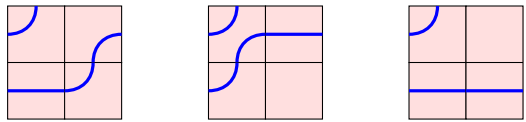
\[
\begin{tikzpicture}[baseline=-20pt]\node[loop] (l) {\plaq{j}&\plaq{r}\\\plaq{h}&\plaq{j}\\};
\end{tikzpicture}
\qquad
\begin{tikzpicture}[baseline=-20pt]\node[loop] (l) {\plaq{a}&\plaq{h}\\\plaq{j}&\plaq{}\\};
\end{tikzpicture}
\qquad
\begin{tikzpicture}[baseline=-20pt]\node[loop] (l) {\plaq{j}&\plaq{}\\\plaq{h}&\plaq{h}\\};
\end{tikzpicture}
\]
Renaming the variables (i.e., equivariant parameters) $x_i$, $i=1,\ldots,k$ for rows, $y_j$, $j=1,\ldots,n$ for columns, one finds
\begin{multline*}
\csm(B_-w B_+)=
  \big(x_1+x_2-y_1-y_2\big)
  +\big(x_2^2+x_1 x_2+y_2^2+y_1 y_2-x_2 y_1-2 x_2 y_2-x_1 y_2\big)
\\+\big(x_2 y_1 y_2+x_1 y_1 y_2+x_1 x_2^2-y_1 y_2^2-x_2^2
  y_2+x_2 y_2^2-x_1 x_2 y_1-x_1 x_2 y_2\big)
\end{multline*}
\end{Example}

In particular if $k=n$ and $w$ is a full permutation, we recover
the class of GPDs defined in \S\ref{sec:intro}, and the polynomials $G_w$ up to the sign $(-1)^{\dim(B_- w B_+)}$.

\section{Extensions to $K$-theory}
Our results admit (in some cases conjectural) $K$-theoretic versions. For instance, one can obtain using similar arguments as in \S\ref{sec:SM}
the \defn{motivic Chern class} of an open Kazhdan--Lusztig variety; all the theorems of \S\ref{ssec:KL} thus generalize,
using the same GPDs,
provided their weights are replaced with their $K$-theoretic analogues
\def\plaqctrc{\begin{tikzpicture}[baseline={([yshift=-\the\dimexpr\fontdimen22\textfont2\relax]current  bounding  box.center)}]\plaq{c}\path (plaq.west) --++(.1,.15) node {$\ss i$}; \path (plaq.south) --++(.1,.1) node{$\ss j$};\end{tikzpicture}}
\def\plaqctra{\begin{tikzpicture}[baseline={([yshift=-\the\dimexpr\fontdimen22\textfont2\relax]current  bounding  box.center)}]\plaq{a}\path (plaq.west) --++(.1,.15) node {$\ss i$}; \path (plaq.south) --++(.1,.1) node{$\ss j$};\end{tikzpicture}}

\[
\widetilde{wt}_K(\delta) = \prod_{\setlength{\loopcellsize}{0.4cm}\tikz[baseline=0pt]{\path (0,0.5);\plaq{}\path (plaq.east) --++(.1,0) node {$\ss x$};\path (plaq.north) --++(0,.15) node {$\ss y$};}}
\begin{cases}
  t^{\left[\setlength{\loopcellsize}{0.4cm}\plaqctr{v}\text{ or }i>j\right]}(1-y/x) & \text{ if } \plaqctrc\ \plaqctr{h}\ \plaqctr{v}\\[2mm]
  1-t\,y/x & \text{ if } \plaqctr{}\\[2mm]
  (1-t)(y/x)^{\left[\setlength{\loopcellsize}{0.4cm}\plaqctr{j}\text{ or }i<j\right]}  & \text{ if }\plaqctra\ \plaqctr{j}\ \plaqctr{r} \\[2mm]
\end{cases}
\]
  This allows to recover various formul\ae\ as special cases, e.g., those of \cite{RW-squarezero}.

Conjecturally, $K$-classes of lower-upper varieties should be given by $K$-theoretic GPD polynomials as well;
see in particular \cite[theorem~3]{artic81} for such a conjectural formula for the $K$-class of the commuting variety.

\gdef\MRshorten#1 #2MRend{#1}%
\gdef\MRfirsttwo#1#2{\if#1M%
MR\else MR#1#2\fi}
\def\MRfix#1{\MRshorten\MRfirsttwo#1 MRend}
\renewcommand\MR[1]{\relax\ifhmode\unskip\spacefactor3000 \space\fi
\MRhref{\MRfix{#1}}{{\scriptsize \MRfix{#1}}}}
\renewcommand{\MRhref}[2]{%
\href{http://www.ams.org/mathscinet-getitem?mr=#1}{#2}}
\bibliographystyle{amsalphahyper}
\bibliography{biblio}

\end{document}